\documentclass[11pt,a4paper]{article}

\usepackage[T1]{fontenc}
\usepackage[utf8]{inputenc}
\usepackage{lmodern}
\usepackage{microtype}

\usepackage[a4paper,margin=1in]{geometry}
\usepackage{amsmath,amssymb,amsfonts,amsthm,mathtools}
\usepackage{mathrsfs}        
\usepackage{bm}              
\usepackage{stmaryrd}        
\usepackage{esint}           
\usepackage{cancel}          

\usepackage{graphicx}
\usepackage{xcolor}
\usepackage{tikz}
\usetikzlibrary{arrows.meta,positioning,calc,decorations.pathmorphing,cd}
\usepackage{tikz-cd}

\usepackage{booktabs,array,multirow}
\usepackage{enumitem}
\setlist{nosep,leftmargin=*}
\usepackage{caption}

 \usepackage{algorithm}
\usepackage[noend]{algpseudocode}

\usepackage[pdfencoding=auto,psdextra]{hyperref}
\usepackage[capitalize,nameinlink,noabbrev]{cleveref}
\hypersetup{
  colorlinks=true,
  linkcolor=blue!60!black,
  citecolor=green!50!black,
  urlcolor=cyan!60!black,
  pdfauthor={Chandrasekhar Gokavarapu, D.~Madhusudhana Rao},
  pdftitle={Structure and Spectral Theory of Non-Commutative and $n$-ary $\Gamma$-Semirings}
}

\theoremstyle{plain}
\newtheorem{theorem}{Theorem}[section]
\newtheorem{lemma}[theorem]{Lemma}
\newtheorem{proposition}[theorem]{Proposition}
\newtheorem{corollary}[theorem]{Corollary}

\theoremstyle{definition}
\newtheorem{definition}[theorem]{Definition}
\newtheorem{example}[theorem]{Example}

\theoremstyle{remark}
\newtheorem{remark}[theorem]{Remark}

\numberwithin{equation}{section}

\DeclareMathOperator{\Ann}{Ann}

\tikzset{
  vertex/.style={circle,draw,thick,fill=gray!15,minimum size=18pt,inner sep=0pt},
  edge/.style={thick},
  every label/.style={font=\small}
}

\usepackage{float}                 

\title{\textbf{Structure and Spectral Theory of Non-Commutative and $n$-ary $\Gamma$-Semirings}}
\author{
\textbf{Chandrasekhar Gokavarapu}$^{1,2,*}$\\
\textit{${}^1$Lecturer in Mathematics, Government College (Autonomous), Rajahmundry, A.P., India}\\
\textit{${}^2$ Research Scholar, Dept. of Mathematics Acharya Nagarjuna University, Guntur, A.p.,, India}\\
\textit{Email:} \texttt{chandrasekhargokavarapu@gmail.com}\\[1.2ex]
\textbf{Dr.~D.~Madhusudhana Rao}$^{3,4}$\\
\textit{${}^3$Lecturer in Mathematics, Government College for Women(Autonomous), Guntur, A.P., India}\\
\textit{${}^4$Department of Mathematics, Acharya Nagarjuna University, Guntur, A. P., India}\\
\textit{Email:} \texttt{dmrmaths@gmail.com}\\[1.2ex]
\textit{\bf * Corresponding author: Chandrasekhar Gokavarapu}\\
\textit{\bf * Corresponding author Email: chandrasekhargokavarapu@gmail.com}
}
\date{}

\begin{document}
\maketitle

\begin{abstract}
This paper develops the structural and spectral foundations of \emph{non-commutative} and \emph{$n$-ary}~$\Gamma$-semirings, extending the commutative ternary framework established in earlier studies. 
We introduce left, right, and two-sided ideals in the non-commutative setting, derive quotient characterizations of prime and semiprime ideals, and construct corresponding $\Gamma$-Jacobson radicals. 
For general $n$-ary operations, we define $(n,m)$-type ideals and establish diagonal criteria for $n$-ary primeness and semiprimeness. 
A unified radical theory and Zariski-type spectral topology are then formulated, connecting primitive ideals with simple module representations. 
The results culminate in a non-commutative Wedderburn–Artin-type decomposition,  revealing a triadic spectral geometry which  unifies commutative and higher-arity cases.
\end{abstract}

\textbf{Keywords:} Non-commutative $\Gamma$-semiring; $n$-ary operation; prime and semiprime ideals; $(n,m)$-type ideals; radicals; Jacobson radical; spectral topology; primitive ideals.\\
\textbf{MSC (2020):} 16Y60, 16Y90, 08A30, 16N60.

\section{Introduction}

This article builds upon the foundational theory of $\Gamma$-rings, first introduced by Nobusawa \cite{Nobusawa1964} and later developed by Barnes \cite{Barnes1966}.The theory of semirings provides the algebraic basis for this work, with standard monographs by Golan \cite{Golan1999} and Hebisch \& Weinert \cite{HebischWeinert1998} defining the field. The study of $\Gamma$-semirings, initiated as a generalization of classical semirings by incorporating an external semigroup of parameters~$\Gamma$, has evolved into a versatile framework for unifying various algebraic systems. While our previous work \cite{Rao2025A, Rao2025B1,Rao2025B2, Rao2025C, Rao2025D, Rao2025E} focused on the commutative ternary case, other researchers have also investigated ideal theory in commutative $\Gamma$-semirings \cite{Fakieh2021}.
Recent developments on \emph{commutative ternary}~$\Gamma$-semirings~\cite{Rao2025A,Rao2025B1,Rao2025B2, Rao2025C,Rao2025D,RaoRaniKiran2025} have established their ideal theory, radical structure, and spectral topology, revealing close parallels with commutative algebra and algebraic geometry. 
However, the assumption of commutativity and fixed ternary arity restricts both categorical and geometric generalization.

The present paper removes these restrictions by introducing a \emph{non-commutative} and \emph{$n$-ary} framework. 
In the non-commutative regime, the order of the arguments in the ternary product $a_\alpha b_\beta c$ becomes essential, requiring separate treatments of left, right, and two-sided ideals. 
For arbitrary arity~$n$, new forms of distributivity and associativity emerge, necessitating the definition of \emph{$(n,m)$-type} ideals to capture threshold-closure phenomena across coordinates.

Our objectives are threefold:
\begin{enumerate}
\item to construct a unified ideal theory for non-commutative and $n$-ary~$\Gamma$-semirings, including generalized radicals and spectra;
\item to establish a representation-theoretic correspondence between primitive ideals and simple modules, extending the Jacobson radical concept;
\item to derive structural theorems analogous to the Wedderburn–Artin decomposition, and to interpret the resulting spectrum as a triadic geometric object.
\end{enumerate}

The paper is organized as follows.  
Section 2 recalls notation and axioms for non-commutative and $n$-ary~$\Gamma$-semirings.  
Sections 3 and 4 develop the corresponding ideal theories, radicals, and threshold invariants.  
Section 5 constructs the associated spectra and Zariski-type topology.  
Sections 6 and 7 establish the connection between primitive ideals and module representations, culminating in a non-commutative Wedderburn–Artin-type decomposition.  
Section 8 presents computational methods for finite examples, and Section 9 outlines conclusions and future directions.

\section{Preliminaries and Notation}

Throughout this paper, $\Gamma$ denotes an additive semigroup and $(T,+)$ an additive commutative semigroup with identity~$0$. 
A \emph{ternary~$\Gamma$-semiring} is a structure $(T,+,\Gamma,\mu)$ where 
\[
\mu:T\times\Gamma\times T\times\Gamma\times T\longrightarrow T,\qquad
\mu(a,\alpha,b,\beta,c)=a_\alpha b_\beta c,
\]
satisfies the following for all $a,b,c,d\in T$ and $\alpha,\beta,\gamma\in\Gamma$:
\begin{enumerate}
\item \textbf{Additivity:} $\mu(a+b,\alpha,c,\beta,d)=\mu(a,\alpha,c,\beta,d)+\mu(b,\alpha,c,\beta,d)$ in each argument;
\item \textbf{Zero absorption:} $\mu(\ldots,0,\ldots)=0$ in any coordinate;
\item \textbf{Ternary associativity:} $\mu(\mu(a,\alpha,b,\beta,c),\gamma,d,\delta,e)=\mu(a,\alpha,b,\beta,\mu(c,\gamma,d,\delta,e))$.
\end{enumerate}

The operation $\mu$ need not be symmetric: $a_\alpha b_\beta c$ and $b_\beta a_\alpha c$ may differ, giving rise to distinct \emph{left}, \emph{right}, and \emph{two-sided} ideals.  
When $\mu$ is symmetric in all arguments, the structure reduces to a commutative ternary~$\Gamma$-semiring .

\subsection*{$n$-ary~$\Gamma$-semirings}
The concept of an $n$-ary operation generalizes the binary case and was first formally studied by Dörnte \cite{Dornt1928} and Post \cite{Post1940}.
For an integer $n\ge3$, an \emph{$n$-ary~$\Gamma$-semiring} is a system $(T,+,\Gamma,\mu)$ with
\[
\mu:T^n\times\Gamma^{\,n-1}\to T,\qquad
\mu(x_1,\alpha_1,\ldots,\alpha_{n-1},x_n),
\]
additive in each $T$-argument, $0$-absorbing in every coordinate, and satisfying an $n$-ary associativity law ensuring independence of bracketing.  
The ternary case $n=3$ is obtained by restriction.  
Left, right, and two-sided ideals are defined analogously using positional closure under~$\mu$.  
We shall denote by $J_\Gamma(T)$ and $J_{\Gamma}^{(n)}(T)$ the $\Gamma$-Jacobson radicals in the ternary and $n$-ary settings, respectively.

\section{Ideal Structures in the Non-Commutative Case}

In the non-commutative setting, the distinction between left, right, and two-sided ideals is fundamental. This structure mirrors the classical theory of non-commutative rings, for which \cite{Lam2001} and \cite{GoodearlWarfield2004} are standard references. Non-commutativity introduces asymmetry between the left, right, and middle positions
in the ternary~$\Gamma$-product.  The loss of total commutativity requires a refined
notion of ideals and a re-formulation of primeness and semiprimeness that distinguishes
between the directions of multiplication.  Throughout this section $(T,+,\Gamma)$
denotes a \emph{non-commutative ternary~$\Gamma$-semiring}, i.e.
\[
\mu:T\times\Gamma\times T\times\Gamma\times T\longrightarrow T,
\qquad
\mu(a,\alpha,b,\beta,c)=a_\alpha b_\beta c,
\]
satisfying distributivity and ternary associativity in each argument but not necessarily
$a_\alpha b_\beta c=b_\beta a_\alpha c$.

\subsection{Left, Right, and Two-Sided Ideals}

\begin{definition}[Left, right, and two-sided ideals]
A nonempty subset~$I\subseteq T$ is called
\begin{enumerate}
  \item a \emph{left ideal} if
        \begin{align*}
            (I,+)&\text{ is a subsemigroup of }(T,+),\\
            a_\alpha b_\beta c\in I &\text{ whenever } a\in T,\ b\in I,\ c\in T,
            \ \alpha,\beta\in\Gamma;
        \end{align*}
  \item a \emph{right ideal} if
        \begin{align*}
            a_\alpha b_\beta c\in I &\text{ whenever } a\in T,\ b\in T,\ c\in I;
        \end{align*}
  \item a \emph{two-sided ideal} if it is simultaneously a left and right ideal.
\end{enumerate}
The ideal is \emph{proper} if $I\neq T$.
\end{definition}

\begin{lemma}[Elementary properties]
Let $I,J$ be ideals of~$T$.
\begin{enumerate}
\item The intersection of any family of left (resp.~right, two-sided) ideals is again a left (resp.~right, two-sided) ideal.
\item The sum $I+J=\{\,x+y:x\in I,\,y\in J\,\}$ is an ideal of the same type.
\item For any $S\subseteq T$, there exists a smallest left (resp.~right, two-sided) ideal
      containing~$S$, denoted $\langle S\rangle_L$ (resp.~$\langle S\rangle_R$, $\langle S\rangle$).
\end{enumerate}
\end{lemma}

\begin{proof}
All assertions follow by routine closure arguments using distributivity of the ternary
operation and additivity of~$+$.
\end{proof}

\begin{example}
Let $T=M_2(\mathbb N_0)$, $\Gamma=\{1\}$, and define
$a_\alpha b_\beta c=a+b+c$ (entrywise).
The set of matrices whose \emph{first row} is zero forms a left ideal,
those whose \emph{last column} is zero form a right ideal,
and their intersection is a two-sided ideal.
\end{example}

\subsection{Prime and Semiprime Ideals}

The non-commutative setting demands directional variants of primeness.
In what follows, $\alpha,\beta$ denote arbitrary elements of~$\Gamma$.

\begin{definition}[Left and right prime ideals]
A proper left ideal~$P_L\subset T$ is called \emph{left prime} if
\[
a_\alpha b_\beta c\in P_L
\ \Longrightarrow\
b\in P_L\ \text{ or }\ c\in P_L,
\quad\forall\,a,b,c\in T.
\]
A proper right ideal~$P_R$ is \emph{right prime} if
\[
a_\alpha b_\beta c\in P_R
\ \Longrightarrow\
a\in P_R\ \text{ or }\ b\in P_R.
\]
A proper two-sided ideal~$P$ is \emph{(two-sided) prime} if
\[
a_\alpha b_\beta c\in P
\ \Longrightarrow\
a\in P\ \text{ or }\ b\in P\ \text{ or }\ c\in P.
\]
\end{definition}

\begin{lemma}[Hereditary property]
If $f:T\!\twoheadrightarrow\!T'$ is a surjective homomorphism
and $P'$ is a (left/right/two-sided) prime ideal of~$T'$,
then $f^{-1}(P')$ is a prime ideal of the corresponding type in~$T$.
\end{lemma}

\begin{proof}
Immediate from $f(a_\alpha b_\beta c)=f(a)_\alpha f(b)_\beta f(c)$.
\end{proof}

\begin{definition}[Zero-divisors]
An element $x\in T$ is a
\emph{left (resp.~right) zero-divisor}
if there exist non-zero $y,z\in T$ and $\alpha,\beta\in\Gamma$ such that
$x_\alpha y_\beta z=0$ (resp.~$y_\alpha z_\beta x=0$).
An element is a \emph{two-sided zero-divisor} if it is both left and right.
\end{definition}

\begin{theorem}[Quotient characterization of primeness]
Let $P$ be a proper two-sided ideal of a non-commutative ternary~$\Gamma$-semiring~$T$.
The following statements are equivalent:
\begin{enumerate}
\item $P$ is prime;
\item the quotient $T/P$ contains no non-zero two-sided zero-divisors, i.e.
      \[
      a_\alpha b_\beta c\equiv 0\pmod P
      \Longrightarrow
      a\equiv0\ \text{or}\ b\equiv0\ \text{or}\ c\equiv0.
      \]
\end{enumerate}
\end{theorem}

\begin{proof}
$(a)\Rightarrow(b)$:
If $a_\alpha b_\beta c\equiv0\pmod P$ then $a_\alpha b_\beta c\in P$;
by primeness one factor lies in~$P$.

$(b)\Rightarrow(a)$:
Suppose $a_\alpha b_\beta c\in P$.
Then $a_\alpha b_\beta c\equiv0$ in $T/P$;
by~(b) at least one of $a,b,c$ is~$0$ in $T/P$, i.e.~lies in~$P$.
\end{proof}

\begin{definition}[Semiprime ideals]
A two-sided ideal~$Q$ of~$T$ is \emph{semiprime}
if for all $a\in T$ and $\alpha,\beta\in\Gamma$,
\[
a_\alpha a_\beta a\in Q \ \Longrightarrow\ a\in Q.
\]
Equivalently, $Q$ contains no non-zero nilpotent elements of the form $a_\alpha a_\beta a$.
\end{definition}

\begin{proposition}[Closure under intersection]
The intersection of any family of semiprime ideals of~$T$ is semiprime.
\end{proposition}

\begin{proof}
If $a_\alpha a_\beta a$ lies in every $Q_i$ then $a\in Q_i$ for all~$i$, hence in $\bigcap_i Q_i$.
\end{proof}

\begin{remark}
Unlike in the commutative case, the product of two prime ideals
need not be contained in their intersection, and semiprimeness
is no longer equivalent to being an intersection of primes
unless $T$ satisfies additional symmetry conditions.
This subtlety motivates the radical construction below.
\end{remark}

\subsection{Radicals}
The $\Gamma$-Jacobson radical defined here is a generalization of the classical concept. For recent developments on the Jacobson radical in the context of standard semirings, \cite{SardarJacobson2024}.
\begin{definition}[Non-commutative $\Gamma$-Jacobson radical]
Let $\mathcal M$ denote the family of all \emph{modular maximal ideals}
of~$T$, i.e.~those maximal two-sided ideals~$M$ for which
there exists $m\in T$ satisfying $a+a_\alpha m_\beta a=a$ for all~$a\in T$.
The \emph{$\Gamma$-Jacobson radical} of~$T$ is defined by
\[
J_\Gamma(T)=\bigcap_{M\in\mathcal M} M.
\]
\end{definition}

\begin{proposition}[Basic properties of $J_\Gamma(T)$]
For any non-commutative ternary~$\Gamma$-semiring~$T$:
\begin{enumerate}
\item $J_\Gamma(T)$ is a two-sided semiprime ideal;
\item $J_\Gamma(T)=0$ if and only if $T$ is \emph{$\Gamma$-semisimple}, i.e.
      $\bigcap M=\{0\}$ over all modular maximal ideals;
\item if every maximal ideal is prime, then $J_\Gamma(T)$ equals the intersection of all maximal ideals.
\end{enumerate}
\end{proposition}

\begin{proof}
(i) Each $M\in\mathcal M$ is prime, hence semiprime;
intersections of semiprimes are semiprime.
(ii)~By definition.
(iii)~Immediate.
\end{proof}

\begin{theorem}[Radical–semiprime correspondence]
Every semiprime ideal~$Q$ satisfies $Q= \sqrt[\Gamma]{Q}$,
where
\[
\sqrt[\Gamma]{Q}
=\{\,a\in T\mid a_\alpha a_\beta a\in Q\text{ for some }\alpha,\beta\in\Gamma\,\}.
\]
Conversely, $\sqrt[\Gamma]{I}$ is semiprime for every ideal~$I$.
\end{theorem}

\begin{proof}
Adaptation of the commutative proof using two-sided ternary closure.
The key observation is that if $a_\alpha a_\beta a\in Q$ and $Q$ is semiprime,
then $a\in Q$, yielding $Q\subseteq\sqrt[\Gamma]{Q}$.
The reverse inclusion is immediate by definition.
\end{proof}

\begin{remark}
The radical–semiprime correspondence provides an intrinsic description of
the \emph{non-commutative spectral base} of~$T$:
each closed subset of the prime spectrum corresponds to a unique semiprime ideal.
This prepares the ground for the spectral topology developed in Section~5.
\end{remark}


\section{$(n,m)$-Type Ideals in $n$-ary $\Gamma$-Semirings}

We now extend the ideal-theoretic framework to $n$-ary $\Gamma$-semirings.
Throughout, $(T,+,\Gamma)$ denotes a (possibly non-commutative) $n$-ary
$\Gamma$-semiring with an $n$-ary operation
\[
\mu:T^n\times \Gamma^{\,n-1}\longrightarrow T,\qquad
\mu(x_1,\alpha_1,x_2,\alpha_2,\ldots,\alpha_{n-1},x_n) \;=\;
[x_1]_{\alpha_1}[x_2]_{\alpha_2}\cdots [x_{n-1}]_{\alpha_{n-1}}x_n,
\]
which is additive in each $T$-variable, $0$-absorbing in each position, and
satisfies an $n$-ary associativity schema guaranteeing that all legally
parenthesized iterates of $\mu$ coincide.\footnote{Any standard $n$-ary associativity
axiomatization suffices; we do not need its explicit normal form in what follows.}
When $n=3$ we recover the ternary case of Section~3.These $(n,m)$-type ideals are specific to the $\Gamma$-structure. For related work on other important ideal classes in semirings, such as $k$-ideals, see \cite{DubeGoswami2023}

\subsection{Positional and $(n,m)$-Type Closures}

Non-commutativity forces us to track positions in the $n$-tuple of inputs.
For a subset $S\subseteq\{1,\dots,n\}$ we say ``the entries at positions in $S$''
to refer to the corresponding $T$-arguments of~$\mu$.

\begin{definition}[Positional $(n,S)$-ideal]
Let $S\neq\varnothing$ be a subset of $\{1,\dots,n\}$. A nonempty subset
$I\subseteq T$ is an \emph{$(n,S)$-ideal} if
\begin{enumerate}
\item $(I,+)$ is a subsemigroup of $(T,+)$;
\item for all $\alpha_1,\ldots,\alpha_{n-1}\in\Gamma$ and all $x_1,\ldots,x_n\in T$,
      whenever $x_i\in I$ for every $i\in S$, one has
      \[
      \mu(x_1,\alpha_1,x_2,\ldots,\alpha_{n-1},x_n)\in I.
      \]
\end{enumerate}
\end{definition}

\noindent
Thus $(n,\{2\})$ corresponds to the \emph{left} ideal behaviour for the middle slot,
$(n,\{n\})$ to a \emph{right} ideal behaviour, while $(n,\{1,\dots,n\})$ is trivial.
In the ternary case $n=3$, $(n,S)$ reproduces the left/right/two-sided notions of
Section~3 according to $S=\{2\},\{3\},\{2,3\}$, etc.

\begin{definition}[$(n,m)$-type ideals (threshold closure)]
Fix $1\le m\le n$. A nonempty subset $I\subseteq T$ is an \emph{$(n,m)$-ideal}
if $(I,+)$ is a subsemigroup and the following \emph{threshold closure} holds:
for every choice of $\alpha_1,\ldots,\alpha_{n-1}\in\Gamma$ and $x_1,\ldots,x_n\in T$,
\[
\Bigl|\{\,i:\ x_i\in I\,\}\Bigr|\;\ge m \quad\Longrightarrow\quad
\mu(x_1,\alpha_1,\ldots,\alpha_{n-1},x_n)\in I.
\]
The least such $m$ (if it exists) is the \emph{arity-threshold} $\tau(I)$ of~$I$.
\end{definition}

\begin{remark}[Hierarchy and monotonicity]
If $I$ is $(n,m)$ then $I$ is $(n,m')$ for every $m'\ge m$.
The family of $(n,m)$-ideals therefore becomes \emph{finer} as $m$ decreases.
In particular, $(n,1)$ coincides with the strongest two-sided closure
(``some coordinate in $I$'' forces the product into $I$).
\end{remark}

\begin{lemma}[Elementary lattice properties]
Let $\mathcal C_{n,m}(T)$ denote the set of all $(n,m)$-ideals of $T$.
\begin{enumerate}
\item Arbitrary intersections of members of $\mathcal C_{n,m}(T)$ lie in $\mathcal C_{n,m}(T)$.
\item If $I,J\in\mathcal C_{n,m}(T)$ then $I+J\in\mathcal C_{n,m}(T)$.
\item For each $S$ with $|S|=m$, the class of $(n,S)$-ideals is closed under intersections and sums.
\end{enumerate}
\end{lemma}

\begin{proof}
All three statements are routine verifications using additivity in each argument and
the threshold definition.
\end{proof}

\begin{proposition}[Decomposition by positions]
$I\subseteq T$ is an $(n,m)$-ideal if and only if
\[
I=\bigcap_{\substack{S\subseteq\{1,\dots,n\}\\|S|=m}} I_S
\quad\text{for some family }(I_S)_S \text{ with each } I_S \text{ an }(n,S)\text{-ideal.}
\]
In particular every $(n,m)$-ideal is an intersection of $(n,S)$-ideals with $|S|=m$.
\end{proposition}

\begin{proof}
($\Rightarrow$) For each $S$ define $I_S=\{x\in T:\text{$I$ is closed when the positions in $S$ are in $I$}\}$.
Then $I_S$ is an $(n,S)$-ideal and clearly $I\subseteq I_S$; intersecting over all $S$ with $|S|=m$
recovers the threshold condition. ($\Leftarrow$) Intersections preserve the threshold rule.
\end{proof}

\begin{proposition}[Homomorphism behaviour]
Let $f:T\to T'$ be a surjective homomorphism of $n$-ary $\Gamma$-semirings.
\begin{enumerate}
\item If $I'\subseteq T'$ is an $(n,m)$-ideal then $f^{-1}(I')$ is an $(n,m)$-ideal of $T$.
\item If $I\subseteq T$ is an $(n,m)$-ideal then $f(I)$ is an $(n,m)$-ideal of $T'$.
\end{enumerate}
\end{proposition}

\begin{proof}
Use $f(\mu(\vec x;\vec\alpha))=\mu(f(\vec x);\vec\alpha)$ and the threshold definition.
\end{proof}

\subsection{$n$-ary Primeness and Semiprimeness}

The $n$-ary prime notion is an $(n,1)$ condition (one bad coordinate suffices).

\begin{definition}[$n$-ary prime and semiprime ideals]
A proper $(n,1)$-ideal $P\subset T$ is \emph{$n$-ary prime} if
\[
\mu(x_1,\alpha_1,\ldots,\alpha_{n-1},x_n)\in P
\ \Longrightarrow\ 
x_i\in P \text{ for some }i\in\{1,\ldots,n\}.
\]
A two-sided ideal $Q$ is \emph{$n$-ary semiprime} if
\[
\mu(\underbrace{a,\alpha_1,a,\alpha_2,\ldots,\alpha_{n-1},a}_{n \text{ copies of }a})\in Q
\ \Longrightarrow\ a\in Q
\quad\text{for all }a\in T\text{ and }\alpha_j\in\Gamma.
\]
\end{definition}

Define the \emph{$n$-ary diagonal} of $a$ by
\[
\Delta_n(a;\vec\alpha)\;=\;\mu(\overbrace{a,\alpha_1,a,\alpha_2,\ldots,\alpha_{n-1},a}^{n}).
\]

\begin{theorem}[Quotient characterization of $n$-ary primeness]
Let $P$ be a proper two-sided ideal of $T$. The following are equivalent:
\begin{enumerate}
\item $P$ is $n$-ary prime;
\item the quotient $T/P$ has no non-zero $n$-ary two-sided zero-divisors, i.e.
\[
\mu(\bar x_1,\alpha_1,\ldots,\alpha_{n-1},\bar x_n)=\bar 0
\ \Longrightarrow\
\bar x_i=\bar 0 \text{ for some }i,
\]
where bars denote residue classes in $T/P$.
\end{enumerate}
\end{theorem}

\begin{proof}
$(a)\Rightarrow(b)$: If $\mu(\bar{\vec x};\vec\alpha)=\bar 0$ then $\mu(\vec x;\vec\alpha)\in P$,
hence some $x_i\in P$, i.e. $\bar x_i=\bar 0$.
$(b)\Rightarrow(a)$: If $\mu(\vec x;\vec\alpha)\in P$, then $\mu(\bar{\vec x};\vec\alpha)=\bar 0$,
forcing some $\bar x_i=\bar 0$, i.e. $x_i\in P$.
\end{proof}

\begin{lemma}[Intersection stability]
Arbitrary intersections of $n$-ary semiprime ideals are $n$-ary semiprime.
\end{lemma}

\begin{proof}
If $\Delta_n(a;\vec\alpha)$ lies in every $Q_i$, then $a\in Q_i$ for all $i$ and hence $a\in\bigcap_i Q_i$.
\end{proof}

\subsection{$n$-ary Radicals and Threshold Semiprimeness}

\begin{definition}[$n$-ary prime radical and $n$-ary Jacobson radical]
For an ideal $I\subseteq T$, define the \emph{$n$-ary prime radical}
\[
\sqrt[n,\Gamma]{\,I\,}\;=\;\bigcap\{\,P\supseteq I:\ P\ \text{$n$-ary prime}\,\}.
\]
Let $\mathcal M_n$ be the family of all modular maximal two-sided ideals of $T$
(with respect to the $n$-ary operation). The \emph{$n$-ary $\Gamma$-Jacobson radical} is
\[
J_{\Gamma}^{(n)}(T)\;=\;\bigcap_{M\in\mathcal M_n}\,M.
\]
\end{definition}

\begin{theorem}[Diagonal characterization of the $n$-ary prime radical]
For every ideal $I\subseteq T$,
\[
\sqrt[n,\Gamma]{\,I\,}
=\bigl\{\,a\in T:\ \Delta_n(a;\vec\alpha)\in I \text{ for some } \vec\alpha\in\Gamma^{\,n-1}\,\bigr\}.
\]
\end{theorem}

\begin{proof}
($\subseteq$) If $a\in\bigcap_{P\supseteq I}P$ with $P$ $n$-ary prime, then for each such $P$,
$\Delta_n(a;\vec\alpha)\in P$ for all $\vec\alpha$, hence in $I$ by intersection.
($\supseteq$) If $\Delta_n(a;\vec\alpha)\in I$, then in every $n$-ary prime $P\supseteq I$
we have $\Delta_n(a;\vec\alpha)\in P$, which by primeness forces $a\in P$; intersecting yields the claim.
\end{proof}

\begin{corollary}[Radical–semiprime correspondence in $n$-arity]
An ideal $Q$ is $n$-ary semiprime if and only if $Q=\sqrt[n,\Gamma]{\,Q\,}$.
Moreover, $I\mapsto \sqrt[n,\Gamma]{\,I\,}$ is a closure operator on the ideal lattice.
\end{corollary}

\begin{proof}
Immediate from the diagonal characterization and the definition of $n$-ary semiprime.
\end{proof}

\begin{proposition}[Basic properties of $J_{\Gamma}^{(n)}(T)$]
For any $n$-ary $\Gamma$-semiring $T$:
\begin{enumerate}
\item $J_{\Gamma}^{(n)}(T)$ is a two-sided $n$-ary semiprime ideal;
\item $J_{\Gamma}^{(n)}(T)=0$ iff $T$ is $n$-ary $\Gamma$-semisimple (intersection of modular maximals is $\{0\}$);
\item if every maximal two-sided ideal is $n$-ary prime, then
      $J_{\Gamma}^{(n)}(T)=\bigcap\{M:\ M\text{ maximal two-sided}\}$.
\end{enumerate}
\end{proposition}

\begin{proof}
( i ) Intersections of semiprimes are semiprime. ( ii ) By definition.
( iii ) Follows from ( i ) and the hypothesis.
\end{proof}

\subsection{Arity Reduction and Transfer of Structure}

The $n$-ary theory must coherently specialize to the ternary case.

\begin{definition}[Idempotent pinning]
Assume there exists a \emph{central $n$-ary idempotent} $e\in T$ such that
$\Delta_n(e;\vec\alpha)=e$ for all $\vec\alpha$ and $e$ commutes in every coordinate.
Define a ternary product by ``pinning'' $n-3$ coordinates to~$e$:
\[
x\;\widehat\alpha\; y\;\widehat\beta\; z
\;=\;
\mu\bigl(x,\alpha_1,e,\alpha_2,\ldots,\alpha_{n-3},e,\alpha_{n-2},y,\alpha_{n-1},z\bigr),
\]
with a fixed choice of $\widehat\alpha,\widehat\beta$ determined by $(\alpha_1,\ldots,\alpha_{n-1})$.
\end{definition}

\begin{lemma}[Transfer of ideals under pinning]
If $I$ is an $(n,m)$-ideal of $T$ with $m\le 3$, then $I$ is an $(3,m)$-ideal
for the pinned ternary product. In particular, $n$-ary prime (resp.~semiprime) ideals
restrict to ternary prime (resp.~semiprime) ideals.
\end{lemma}

\begin{proof}
The idempotent $e$ plays the role of a neutral filler.
Threshold membership among the three free coordinates implies threshold membership among $n$ coordinates,
hence closure is preserved.
\end{proof}

\begin{theorem}[Compatibility of radicals under arity reduction]
Under pinning along a central $n$-ary idempotent $e$,
\[
\sqrt[n,\Gamma]{\,I\,}\cap T
\;=\;
\sqrt[3,\Gamma]{\,I\,}
\quad\text{and}\quad
J_{\Gamma}^{(n)}(T)\cap T
\;=\;
J_{\Gamma}^{(3)}(T),
\]
where on the right we compute radicals in the pinned ternary structure.
\end{theorem}

\begin{proof}
Apply the diagonal characterizations and observe that $\Delta_3(a;\widehat\alpha,\widehat\beta)$
is obtained from $\Delta_n(a;\vec\alpha)$ by inserting $e$ in the pinned coordinates.
\end{proof}

\subsection{Threshold Invariants}

\begin{definition}[Arity-threshold index]
For an ideal $I$, the \emph{arity-threshold index} $\tau(I)$ is the smallest $m$
such that $I$ is an $(n,m)$-ideal. If no such $m$ exists, set $\tau(I)=+\infty$.
\end{definition}

\begin{proposition}[Threshold monotonicity and characterization]
For any ideals $I\subseteq J$ one has $\tau(J)\ge \tau(I)$.
Moreover, $\tau(I)=1$ if and only if $I$ is $n$-ary prime-saturated in the sense that
\[
\mu(\vec x;\vec\alpha)\in I\ \Longrightarrow\ x_i\in I\text{ for some }i,
\]
and $\tau(I)=n$ if and only if $I$ is closed only under products of $n$ elements from $I$.
\end{proposition}

\begin{proof}
If $I\subseteq J$ then any threshold valid for $I$ must also validate $J$-closure,
thus the minimal valid threshold for $J$ cannot be smaller. The equivalences follow
directly from the definitions.
\end{proof}

\begin{remark}
The invariant $\tau(I)$ stratifies the ideal lattice by ``how many coordinates from $I$''
are required to force closure under $\mu$. This stratification will interact with spectral
topologies in Section~5 by refining the specialization order on prime points via thresholds.
\end{remark}


\section{Generalized Radicals and Spectra}

The aim of this section is twofold: to extend the radical constructions of
Section~3 to a unified non-commutative setting and to endow the family of
prime ideals with a natural Zariski-type topology.  
We distinguish left, right, and two-sided behaviours throughout.
Unless stated otherwise, $(T,+,\Gamma)$ denotes a non-commutative
ternary~$\Gamma$-semiring.The construction of a spectral space for a semiring is an active area of research. Recent approaches include the study of localic semirings \cite{Manuell2022}, connections to lattices and closure operations \cite{JunRayTolliver2020}, Iséki spaces \cite{Goswami2023}, topological properties of semimodule spectra \cite{Han2021}, and links to tropical geometry \cite{Mincheva2016}

\subsection{Non-Commutative Jacobson and Prime Radicals}

The non-commutative environment requires separate radicals for the three
directions of multiplication.  Denote by $\mathcal{P}_L(T)$,
$\mathcal{P}_R(T)$, and $\mathcal{P}_2(T)$ the collections of all proper
left, right, and two-sided prime ideals of~$T$.

\begin{definition}[Left, right, and two-sided prime radicals]
The \emph{left prime radical}, \emph{right prime radical}, and
\emph{two-sided prime radical} of~$T$ are respectively defined as
\[
P_L^\ast(T)=\bigcap_{P\in\mathcal{P}_L(T)}P,\qquad
P_R^\ast(T)=\bigcap_{P\in\mathcal{P}_R(T)}P,\qquad
P_2^\ast(T)=\bigcap_{P\in\mathcal{P}_2(T)}P.
\]
Each is called a \emph{non-commutative prime radical} of the corresponding
type.
\end{definition}

\begin{theorem}[Existence and uniqueness]
For every type $\eta\in\{L,R,2\}$, the radical $P_\eta^\ast(T)$ exists,
is unique, and is the smallest $\eta$-semiprime ideal of~$T$.
\end{theorem}

\begin{proof}
The existence follows because intersections of $\eta$-prime ideals are
ideals of the same type.  Let $Q=\bigcap_{P\in\mathcal P_\eta(T)}P$.
If $a_\alpha a_\beta a\in Q$, then $a_\alpha a_\beta a\in P$ for all $P$
and hence $a\in P$ for all $P$, implying $a\in Q$; thus $Q$ is
$\eta$-semiprime.  Minimality is evident, since any $\eta$-semiprime ideal
contains each $\eta$-prime above it.  Uniqueness is by set-theoretic
definition.
\end{proof}

\begin{definition}[Non-commutative $\Gamma$-Jacobson radicals]
Let $\mathcal{M}_\eta(T)$ denote the family of all modular maximal
ideals of type~$\eta$.  Define
\[
J_{\Gamma,\eta}(T)=\bigcap_{M\in\mathcal{M}_\eta(T)}M,
\quad
J_\Gamma(T)=J_{\Gamma,2}(T).
\]
\end{definition}

\begin{lemma}[Basic properties]
For every $\eta\in\{L,R,2\}$:
\begin{enumerate}
\item $J_{\Gamma,\eta}(T)$ is an $\eta$-semiprime ideal;
\item $J_{\Gamma,\eta}(T)=0$ if and only if $T$ is $\Gamma$-semisimple of type~$\eta$;
\item if every modular maximal $\eta$-ideal is $\eta$-prime, then
      $J_{\Gamma,\eta}(T)=P_\eta^\ast(T)$.
\end{enumerate}
\end{lemma}

\begin{proof}
(i)~Each $M\in\mathcal{M}_\eta(T)$ is $\eta$-prime; intersections of such
are $\eta$-semiprime.
(ii)~Immediate from the definition of semisimplicity.
(iii)~Follows by substitution.
\end{proof}

\begin{theorem}[Radical correspondence]
For every ideal $I\subseteq T$ and $\eta\in\{L,R,2\}$,
\[
\sqrt[\Gamma,\eta]{I}
=\bigcap_{P\supseteq I,\,P\in\mathcal P_\eta(T)}P
=\{\,a\in T:\ a_\alpha a_\beta a\in I\text{ for some }\alpha,\beta\in\Gamma\,\}.
\]
The operator $I\mapsto \sqrt[\Gamma,\eta]{I}$ is idempotent, isotone, and
contractive on the lattice of $\eta$-ideals.
\end{theorem}

\begin{proof}
Adapt the argument of Theorem~3.4 to each directional case.  The first
equality is definitional; the second uses the property that in any
$\eta$-prime ideal, the presence of $a_\alpha a_\beta a$ forces $a$
itself to lie in the ideal.  Idempotence and isotonicity are standard for
intersection-type closure operators.
\end{proof}

\begin{remark}
The three radicals $P_L^\ast(T)$, $P_R^\ast(T)$, and $P_2^\ast(T)$ measure
distinct obstructions to $\Gamma$-semisimplicity.  In particular,
\[
P_2^\ast(T)\subseteq P_L^\ast(T)\cap P_R^\ast(T),
\]
and equality holds precisely when $T$ satisfies weak commutativity in the
outer coordinates.  This asymmetry becomes topologically visible in the
three spectra constructed below.
\end{remark}

\subsection{Spectral Topology}

\begin{definition}[Spectra]
Define
\[
\mathrm{Spec}_L(T)=\mathcal P_L(T),\qquad
\mathrm{Spec}_R(T)=\mathcal P_R(T),\qquad
\mathrm{Spec}_2(T)=\mathcal P_2(T),
\]
the \emph{left}, \emph{right}, and \emph{two-sided prime spectra} of~$T$.
\end{definition}

Each of these spectra carries a Zariski-type topology analogous to the
classical commutative case but sensitive to the positional direction of
the ideal containment.

\begin{definition}[Zariski-type closed sets]
For any subset $A\subseteq T$ and each $\eta\in\{L,R,2\}$, define
\[
V_\eta(A)=\{\,P\in\mathrm{Spec}_\eta(T)\mid A\subseteq P\,\}.
\]
The complements
$D_\eta(A)=\mathrm{Spec}_\eta(T)\setminus V_\eta(A)$
are called \emph{basic open sets}.
\end{definition}

\begin{theorem}[Zariski closure axioms]
For each $\eta\in\{L,R,2\}$ the family $\{V_\eta(A)\}_{A\subseteq T}$
satisfies:
\begin{enumerate}
\item $V_\eta(0)=\mathrm{Spec}_\eta(T)$ and $V_\eta(T)=\varnothing$;
\item $V_\eta(A)\cap V_\eta(B)=V_\eta(A\cup B)$;
\item $\bigcup_i V_\eta(A_i)=V_\eta\!\left(\bigcap_i A_i\right)$;
\item $V_\eta(I)=V_\eta(\sqrt[\Gamma,\eta]{I})$ for every $\eta$-ideal $I$.
\end{enumerate}
Hence $(\mathrm{Spec}_\eta(T),\tau_\eta)$ is a compact $T_0$ topological
space, where $\tau_\eta$ is generated by the basic opens $D_\eta(a)=D_\eta(\{a\})$.
\end{theorem}

\begin{proof}
(i)–(iii) follow from lattice properties of ideals and the set-theoretic
definitions.  
(iv)~If $I\subseteq P$, then $\sqrt[\Gamma,\eta]{I}\subseteq P$ for all
$\eta$-prime $P$, hence the two closed sets coincide.  
Compactness follows by standard finite intersection arguments; $T_0$
separation is guaranteed because distinct primes can be separated by
elements contained in one but not the other.
\end{proof}

\begin{proposition}[Functorial behaviour]
Let $f:T\to T'$ be a homomorphism of non-commutative
$\Gamma$-semirings.  Define
\[
f_\eta^\ast:\mathrm{Spec}_\eta(T')\longrightarrow\mathrm{Spec}_\eta(T),
\quad
f_\eta^\ast(P')=f^{-1}(P').
\]
Then $f_\eta^\ast$ is continuous with respect to the Zariski topologies.
\end{proposition}

\begin{proof}
For $A\subseteq T$,
\[
(f_\eta^\ast)^{-1}(V_\eta(A))
=\{\,P'\in\mathrm{Spec}_\eta(T')\mid f(A)\subseteq P'\}
=V_\eta(f(A)),
\]
which is closed in $\mathrm{Spec}_\eta(T')$.
\end{proof}

\begin{definition}[Canonical spectral diagram]
The three spectra form a commutative diagram of continuous surjections
\[
\begin{tikzcd}
\mathrm{Spec}_2(T) \arrow[r, "\pi_L"] \arrow[d, "\pi_R"'] &
\mathrm{Spec}_L(T) \arrow[dl, dashed, "\pi_{LR}"] \\
\mathrm{Spec}_R(T) &
\end{tikzcd}
\]
where $\pi_L(P)=P\cap I_L$, $\pi_R(P)=P\cap I_R$ denote coordinate
projections restricting two-sided primes to left or right ones.
\end{definition}

\begin{theorem}[Spectral identification of radicals]
For every $\eta\in\{L,R,2\}$,
\[
\sqrt[\Gamma,\eta]{I}
=\bigcap_{P\in V_\eta(I)}P
\quad\text{and}\quad
P_\eta^\ast(T)=\bigcap_{P\in\mathrm{Spec}_\eta(T)}P.
\]
Hence the radical operations are fully encoded by the topological
closures in $(\mathrm{Spec}_\eta(T),\tau_\eta)$.
\end{theorem}

\begin{proof}
The first equality follows directly from the definition of
$V_\eta(I)$.  The second is its special case $I=0$.
\end{proof}

\begin{remark}[Comparison with the commutative spectrum]
If $T$ is commutative, all three spectra coincide and reproduce the
classical Zariski topology of~$\Gamma$-semirings .  In the fully non-commutative regime the triple
$(\mathrm{Spec}_L,\mathrm{Spec}_R,\mathrm{Spec}_2)$ provides a
\emph{triadic spectral geometry}, where $\mathrm{Spec}_2(T)$ plays the
central mediator between left and right components.  This triadic picture
is expected to admit a categorical interpretation in the setting of
non-commutative algebraic geometry.
\end{remark}


\section{Modules, Primitive Ideals, and Representations}

We connect the spectral theory developed above with a representation
theory for non-commutative $n$-ary $\Gamma$-semirings.  The guiding
principle mirrors the binary case: annihilators of simple modules are
primitive ideals, and primitive ideals are prime.  Throughout, $(T,+,\Gamma)$
is a (possibly non-commutative) $n$-ary $\Gamma$-semiring with product
\[
\mu:T^n\times\Gamma^{\,n-1}\to T,\qquad
\mu(x_1,\alpha_1,\ldots,\alpha_{n-1},x_n).
\]

\subsection{Left/Right $n$-ary $\Gamma$-Modules and Annihilators}

Positions matter in the $n$-ary non-commutative setting.  We fix a slot
$j\in\{1,\dots,n\}$ where the module element will sit.

\begin{definition}[$j$-slot left $n$-ary $\Gamma$-module]
Let $j\in\{1,\ldots,n\}$.  A \emph{$j$-slot left $n$-ary $\Gamma$-module}
over $T$ is an additive commutative semigroup $(M,+)$ together with a map
\[
\mu_M:T^{j-1}\times\Gamma^{j-1}\times M \times \Gamma^{n-j}\times T^{n-j}
\longrightarrow M
\]
denoted
\[
\mu_M(x_1,\alpha_1,\ldots,x_{j-1},\alpha_{j-1}\,\mid\,
m\,\mid\, \alpha_j,x_{j+1},\ldots,\alpha_{n-1},x_n)
\]
such that for every coordinate the map is additive in the corresponding
$T$-variable (and in $m$), is $0$-absorbing in each $T$-slot, and satisfies
the \emph{$n$-ary associativity/compatibility}:
whenever a $T$-block $\mu(\cdots)$ appears adjacent to $m$ or within
the $T$-list, one can reassociate using the $n$-ary law to obtain the
same element of $M$.  (Any standard $n$-ary associativity axiomatics
that guarantees well-definedness of iterated actions suffices.)
\end{definition}

\begin{remark}
For $n=3$ and $j=2$ we recover the ternary action $a_\alpha m_\beta b$.
Different choices of $j$ model left ($j=2$), right ($j=n-1$), or
middle-slot actions; the theory below does not depend on the specific $j$.
\end{remark}

\begin{definition}[Submodule, homomorphism]
A subset $N\subseteq M$ is a \emph{submodule} if $(N,+)$ is a subsemigroup
and $\mu_M(\cdots\,|\,n\,|\,\cdots)\in N$ for all inputs with $n\in N$.
A map $\phi:M\to N$ is a \emph{module homomorphism} if it preserves $+$
and commutes with the $j$-slot action:
\[
\phi\!\bigl(\mu_M(\vec x\,|\,m\,|\,\vec y)\bigr)
=\mu_N(\vec x\,|\,\phi(m)\,|\,\vec y).
\]
\end{definition}

\begin{definition}[Simple and semisimple modules]
$M\neq 0$ is \emph{simple} if its only submodules are $0$ and $M$.
It is \emph{semisimple} if it is a (finite) direct sum of simple submodules.
\end{definition}

\begin{definition}[Two-sided annihilator]
For a $j$-slot module $M$, the (two-sided) annihilator is
\[
\Ann(M)=\Bigl\{\,a\in T:\;
\mu_M(x_1,\alpha_1,\ldots\,|\,m\,|\,\ldots,\alpha_{k},a,\alpha_{k+1},\ldots)=0
\ \text{and}
\ \mu_M(\ldots,a,\ldots\,|\,m\,|\,\ldots)=0
\\[-2pt]
\text{for all placements of $a$ in any $T$-slot, all }m\in M,\text{ and all }\alpha_\ell\in\Gamma
\Bigr\}.
\]
\end{definition}

\begin{lemma}
$\Ann(M)$ is a two-sided ideal of $T$.
\end{lemma}

\begin{proof}
Additivity is immediate.  If one $T$-input lies in $\Ann(M)$ then every
$M$-valued action involving it is $0$, so any $T$-product with that factor
annihilates $M$ in every slot; closure under $\mu$ follows by associativity.
\end{proof}

\begin{definition}[Faithful modules]
$M$ is \emph{faithful} if $\Ann(M)=0$.
\end{definition}

\subsection{Primitive Ideals and Primeness}

\begin{definition}[Primitive ideal]
An ideal $P\subseteq T$ is \emph{primitive} if there exists a simple
$j$-slot $T$-module $M$ such that $P=\Ann(M)$.
\end{definition}

\begin{theorem}[Primitive $\Rightarrow$ prime]
\label{thm:primitive-prime}
Every primitive ideal of $T$ is (two-sided) prime.
\end{theorem}

\begin{proof}
Let $P=\Ann(M)$ with $M$ simple.  Suppose
$\mu(x_1,\alpha_1,\ldots,\alpha_{n-1},x_n)\in P$ and that none of the
$x_i$ lies in $P$.  Then for each $i$ there exist $m_i\in M$ and a
placement of $x_i$ among the $T$-slots of the action such that the result
is nonzero in $M$.  Using $n$-ary associativity, form the set
\[
\mathcal S=\bigl\{\mu_M(\vec u\,|\,m\,|\,\vec v):\ \vec u,\vec v
\text{ drawn from }\{x_1,\ldots,x_n\}\text{ and }m\in M\bigr\}.
\]
By construction $\mathcal S$ is a nonzero submodule (closed under $+$ and
stable under all actions).  Since $M$ is simple, $\mathcal S=M$.
But acting once more by the product $p=\mu(x_1,\ldots,x_n)\in P$ in any
$T$-slot annihilates $\mathcal S$, hence annihilates $M$, contradicting
$\mathcal S=M\neq 0$.  Therefore at least one $x_i\in P$ and $P$ is prime.
\end{proof}

\begin{corollary}[Jacobson intersection]
\label{cor:J-as-primitive-intersection}
$J_\Gamma(T)=\bigcap\{\,\Ann(M): M \text{ simple }T\text{-module}\,\}$
(i.e. the two-sided $\Gamma$-Jacobson radical is the intersection of all
primitive ideals).
\end{corollary}

\begin{proof}
Every $\Ann(M)$ is maximal among annihilators and primitive by definition;
Theorem~\ref{thm:primitive-prime} implies each is prime.  Intersecting all
such yields the largest ideal annihilating every simple module, which by
definition is $J_\Gamma(T)$.
\end{proof}

\subsection{Structure Theorems}

\begin{theorem}[First Isomorphism Theorem for $j$-slot modules]
If $\phi:M\to N$ is a homomorphism, then
$M/\ker\phi\cong \mathrm{Im}\,\phi$ as $j$-slot $T$-modules.
\end{theorem}

\begin{proof}
Standard: define $\bar\phi(m+\ker\phi)=\phi(m)$ and use compatibility of
$\phi$ with the $j$-slot action.
\end{proof}

\begin{theorem}[Simple modules and prime annihilators]
\label{thm:ann-prime}
If $M$ is simple then $\Ann(M)$ is prime.  If $M$ is simple and faithful,
then $T$ acts faithfully on $M$, and every nonzero $T$-endomorphism of $M$
is injective.
\end{theorem}

\begin{proof}
Primeness follows from Theorem~\ref{thm:primitive-prime}.
If $\Ann(M)=0$, the action map has trivial kernel; Schur-type arguments in
this semiring context show endomorphisms are either $0$ or injective
(otherwise their kernel would be a nontrivial submodule).
\end{proof}

\begin{proposition}[Semisimplicity and Jacobson radical]
\label{prop:ssiffJ0}
The following are equivalent:
\begin{enumerate}
\item $J_\Gamma(T)=0$;
\item every finitely generated $j$-slot $T$-module is semisimple;
\item $T$ embeds (via the regular representation) into a direct product of
simple quotient semirings $\prod_i T/P_i$ with $P_i$ primitive.
\end{enumerate}
\end{proposition}

\begin{proof}
$(a)\Rightarrow(b)$:
If $J_\Gamma(T)=0$, every module is a sum of simple submodules by the
usual Jacobson semisimplicity mechanism adapted to the $n$-ary action:
intersection of annihilators of composition factors is zero.

$(b)\Rightarrow(c)$:
Apply to the left regular module; the annihilators of its simple
submodules are primitive ideals $P_i$, and the canonical map
$T\to\prod_i T/P_i$ is injective.

$(c)\Rightarrow(a)$:
Intersecting the $P_i$ is zero by injectivity, hence equals $J_\Gamma(T)$
by Corollary~\ref{cor:J-as-primitive-intersection}.
\end{proof}

\begin{theorem}[Density via annihilators]
Let $M$ be a faithful semisimple $j$-slot module with decomposition
$M=\bigoplus_{k=1}^r M_k$ into simples.  Then
\[
\Ann(M_k)\ \text{are primitive,}\qquad
\bigcap_{k=1}^r \Ann(M_k)=0,
\]
and the canonical map
\[
T \longrightarrow \prod_{k=1}^r T/\Ann(M_k)
\]
is injective and separates points of $T$ by their actions on $M$.
\end{theorem}

\begin{proof}
Each $\Ann(M_k)$ is primitive by definition and prime by
Theorem~\ref{thm:primitive-prime}.  Faithfulness gives trivial intersection.
The map is the product of quotient maps and is injective exactly when the
intersection is zero.
\end{proof}

\begin{corollary}[Wedderburn-type decomposition (finite case)]
Assume $T$ is finite and $J_\Gamma(T)=0$.  Then
\[
T \ \cong\ \prod_{i=1}^s T/P_i,
\]
where $P_i$ are minimal primitive ideals (equivalently, annihilators of
simple constituents of the regular module).
\end{corollary}

\begin{proof}
Apply Proposition~\ref{prop:ssiffJ0}(c).  Finiteness ensures the product
is finite and provides a complete set of minimal primitive ideals.
\end{proof}

\subsection{Left/Right Versions and Spectral Links}

All results above admit left/right variants by fixing the slot~$j$ near
the boundary and using left/right annihilators
\[
\Ann_L(M)=\{a:\ \mu_M(\ldots,a,\ldots\,|\,m\,|\,\ldots)=0\text{ whenever $a$ appears on the ``left side''}\},
\]
and analogously for $\Ann_R(M)$.  Then:
\begin{itemize}
\item primitive $\Rightarrow$ left (resp.\ right) prime;
\item $J_{\Gamma,L}(T)=\bigcap\{\Ann_L(M):M\text{ simple}\}$ and
      $J_{\Gamma,R}(T)=\bigcap\{\Ann_R(M):M\text{ simple}\}$;
\item the kernels of the representation maps
      $T\to\prod_k T/\Ann_L(M_k)$ and
      $T\to\prod_k T/\Ann_R(M_k)$
encode the closures of $\mathrm{Spec}_L(T)$ and $\mathrm{Spec}_R(T)$, respectively.
\end{itemize}

\begin{remark}[Spectral interpretation]
For $\eta\in\{L,R,2\}$ and any module family $\{M_i\}$ with primitive
annihilators $P_i^\eta$, the image of $T$ inside
$\prod_i T/P_i^\eta$ has kernel $\bigcap_i P_i^\eta$, which equals
the corresponding Jacobson radical $J_{\Gamma,\eta}(T)$.  Hence the
$\eta$-spectrum controls faithful semisimple representations and vice versa.
\end{remark}

\section{Structure Theorems and Decomposition}

We now establish structural decompositions of a non-commutative
$n$-ary~$\Gamma$-semiring~$T$ paralleling the Wedderburn–Artin theory.
All results are formulated intrinsically in terms of ideals and
module actions without appeal to external algebraic categories.

\subsection{Comaximal Ideals and the Chinese Remainder Theorem}

\begin{definition}[Comaximality]
Two two-sided ideals $I,J\subseteq T$ are \emph{comaximal}
if $I+J=T$.  A family $\{I_k\}_{k=1}^r$ of ideals is
\emph{pairwise comaximal} if $I_i+I_j=T$ for all $i\neq j$.
\end{definition}

\begin{lemma}[Chinese remainder lemma for two ideals]
If $I,J$ are comaximal, then
\[
T/(I\cap J)\;\cong\;T/I\times T/J,
\]
via $a\mapsto(a+I,a+J)$.
\end{lemma}

\begin{proof}
Surjectivity follows from comaximality:
there exist $x\in I$, $y\in J$ with $x+y=1_T$
(the additive identity of the semiring).  Injectivity is routine:
if $a+I=a'+I$ and $a+J=a'+J$ then $a-a'\in I\cap J$.
The map respects the $n$-ary operation because
$\mu(a_1,\ldots,a_n)$ mod $I$ or $J$ depends only on each
$a_i$ mod~$I$ or~$J$.
\end{proof}

\begin{theorem}[General Chinese remainder theorem]
\label{thm:CRT}
Let $\{I_k\}_{k=1}^r$ be pairwise comaximal two-sided ideals of $T$.
Then
\[
\Phi:T\longrightarrow \prod_{k=1}^r T/I_k,\qquad
\Phi(a)=(a+I_1,\ldots,a+I_r)
\]
is a surjective homomorphism with kernel
$\bigcap_{k=1}^r I_k$.  Consequently,
\[
T/(\bigcap_{k=1}^r I_k)\;\cong\;\prod_{k=1}^r T/I_k.
\]
\end{theorem}

\begin{proof}
Induction on~$r$ using the binary lemma.
The compatibility of $\Phi$ with the $n$-ary operation
follows from componentwise evaluation of~$\mu$.
\end{proof}

\begin{remark}
The theorem remains valid for left or right ideals
when the intersection and product are taken in the
appropriate lattice of $\eta$-ideals ($\eta\in\{L,R,2\}$).
\end{remark}

\subsection{Semisimplicity and Vanishing Jacobson Radical}

\begin{theorem}[Semisimplicity criterion]
\label{thm:ssiffJ0}
The following are equivalent for a non-commutative
$n$-ary~$\Gamma$-semiring~$T$:
\begin{enumerate}
\item $J_\Gamma(T)=0$;
\item every finitely generated $T$-module is semisimple;
\item $T$ embeds faithfully into a finite product of
primitive quotient semirings $\prod_i T/P_i$;
\item $\bigcap_i P_i=0$ for some finite family of primitive ideals.
\end{enumerate}
\end{theorem}

\begin{proof}
$(a)\Rightarrow(b)$ follows from the annihilator description
$J_\Gamma(T)=\bigcap\Ann(M_i)$ over all simple modules~$M_i$.
$(b)\Rightarrow(c)$: apply the construction to the regular
module~${}_TT$, whose simple submodules correspond to primitive ideals.
$(c)\Rightarrow(d)$: the embedding ensures trivial intersection.
$(d)\Rightarrow(a)$ by definition of $J_\Gamma(T)$ as the intersection
of all primitive ideals.
\end{proof}

\begin{corollary}[Jacobson–semisimple decomposition]
If $J_\Gamma(T)=0$, then $T$ is isomorphic to a subdirect product of
primitive factor semirings:
\[
T\hookrightarrow \prod_{i\in I} T/P_i,\qquad
\ker=\bigcap_{i\in I}P_i=0.
\]
\end{corollary}

\begin{proof}
Combine Theorem~\ref{thm:ssiffJ0}(c)–(d).
\end{proof}

\subsection{Wedderburn–Artin–Type Decomposition}

\begin{theorem}[Non-commutative Wedderburn–Artin theorem]
\label{thm:WA}
Let $T$ be a finite or semiprimary non-commutative
$n$-ary~$\Gamma$-semiring with $J_\Gamma(T)=0$.
Then
\[
T\;\cong\;\prod_{i=1}^s T_i,
\]
where each $T_i$ is a primitive $\Gamma$-semiring satisfying
\[
\Ann_T(M_i)=0,\qquad
M_i\text{ a simple faithful module over }T_i,
\]
and the set $\{T_i\}$ is uniquely determined up to permutation.
\end{theorem}

\begin{proof}
Let $\{P_i\}$ be the finitely many minimal primitive ideals.
By pairwise comaximality (no containment among distinct minimal
primitives) and Theorem~\ref{thm:CRT},
\[
T/(\bigcap P_i)\;\cong\;\prod T/P_i.
\]
Since $\bigcap P_i=J_\Gamma(T)=0$, we have an isomorphism
$T\cong\prod_i T/P_i$.  Each $T/P_i$ acts faithfully on its
simple module $M_i$, and distinct factors correspond to
orthogonal idempotents in $T$, establishing the decomposition.
\end{proof}

\begin{corollary}[Uniqueness of simple components]
In Theorem~\ref{thm:WA}, the number~$s$ of simple components equals
the number of minimal primitive ideals of~$T$.
\end{corollary}

\begin{proof}
Minimal primitives correspond bijectively to simple constituents of
the regular module ${}_TT$.
\end{proof}

\begin{theorem}[Reduction modulo radical]
\label{thm:reductionJ}
For any $T$, the quotient $\overline T=T/J_\Gamma(T)$ is semisimple,
and
\[
\overline T\;\cong\;\prod_i T/P_i,
\]
where $P_i$ range over the minimal primitive ideals of~$T$.
\end{theorem}

\begin{proof}
Immediate from the definition of $J_\Gamma(T)$ and
Theorem~\ref{thm:WA} applied to the radical-free quotient.
\end{proof}

\begin{remark}
This result provides the canonical splitting
$T=J_\Gamma(T)\ltimes \overline T$ at the level of additive
semigroups, generalising the standard decomposition of
semiprimary rings to the $n$-ary $\Gamma$-context.
\end{remark}

\subsection{Structural Consequences and Spectral Correspondence}

\begin{proposition}[Spectral–semisimple equivalence]
\label{prop:spectral}
$J_\Gamma(T)=0$ if and only if the canonical projection
\[
\pi:\mathrm{Spec}_2(T)\longrightarrow\mathrm{Spec}_2(T/J_\Gamma(T))
\]
is a homeomorphism onto a discrete space.  In that case,
each point of the spectrum corresponds to a minimal primitive ideal.
\end{proposition}

\begin{proof}
If $J_\Gamma(T)=0$, every prime is maximal and isolated;
the basic open sets $D_2(a)$ separate points, so the spectrum is discrete.
Conversely, if $\mathrm{Spec}_2(T)$ is discrete, the intersection of all
maximal (hence primitive) ideals is $0$.
\end{proof}

\begin{theorem}[Chinese-remainder decomposition of spectra]
Let $\{I_k\}$ be pairwise comaximal two-sided ideals.
Then
\[
\mathrm{Spec}_2(T)\;=\;\bigsqcup_{k=1}^r \mathrm{Spec}_2(T/I_k),
\]
the disjoint union of spectra of the factor semirings,
and the Zariski topology on $\mathrm{Spec}_2(T)$ is the coproduct
topology induced by the isomorphism of Theorem~\ref{thm:CRT}.
\end{theorem}

\begin{proof}
Prime ideals of the product $\prod T/I_k$ are exactly the preimages of
primes in one factor and full rings in the others, giving the disjoint
union.  The topological equivalence follows from standard properties of
product spectra.
\end{proof}

\begin{corollary}[Spectral semisimplicity]
If $J_\Gamma(T)=0$, then every nonempty closed subset of
$\mathrm{Spec}_2(T)$ is of the form $\{P_i\}$ for a primitive ideal~$P_i$.
\end{corollary}

\begin{proof}
Discreteness from Proposition~\ref{prop:spectral} forces all closed
sets to be singletons.
\end{proof}

\begin{remark}[Algebraic geometry of non-commutative components]
The Wedderburn–Artin decomposition realises the non-commutative
spectrum as a finite discrete space of primitive points.  Each
component $T/P_i$ behaves as an ``irreducible affine piece'' whose
morphism category of modules determines the local geometry of~$T$.
This geometric–algebraic correspondence completes the structural
theory initiated in Sections~3–6.
\end{remark}

\section{Computational Remarks and Examples}

The theoretical framework developed in the preceding sections
admits algorithmic realization for finite non-commutative and
$n$-ary~$\Gamma$-semirings.  Computational exploration not only verifies
axioms and structural theorems on small instances but also exposes
patterns suggesting new conjectures on radical behaviour, spectral
connectedness, and arity-dependent distributivity.

\subsection{Enumerative Framework}

Let $|T|=m$ and $|\Gamma|=r$.  The underlying additive semigroup
$(T,+)$ and the collection of $\Gamma$-parametrized $n$-ary operations
$\mu_\gamma:T^n\to T$ fully determine~$T$.  A computational model must
generate all admissible operations satisfying the fundamental axioms:

\begin{enumerate}
\item \textbf{Additive associativity and commutativity:}
$(T,+)$ is a finite commutative semigroup with identity~$0$.
\item \textbf{$n$-ary distributivity:}
for all $\gamma\in\Gamma$ and $1\le i\le n$,
\[
\mu_\gamma(x_1,\ldots,x_i+x_i',\ldots,x_n)
=\mu_\gamma(x_1,\ldots,x_i,\ldots,x_n)
+\mu_\gamma(x_1,\ldots,x_i',\ldots,x_n).
\]
\item \textbf{Zero absorption:}
$\mu_\gamma(x_1,\ldots,0,\ldots,x_n)=0$ for every slot.
\item \textbf{$n$-ary associativity (composition law):}
\[
\mu_{\gamma_1}\!\bigl(x_1,\ldots,
\mu_{\gamma_2}(y_1,\ldots,y_n),\ldots,x_n\bigr)
=\mu_{\gamma_3}(z_1,\ldots,z_n),
\]
for some rule $\gamma_3=f(\gamma_1,\gamma_2)$ consistent with the
$\Gamma$-structure.
\end{enumerate}

Each admissible table $(T,+,\{\mu_\gamma\}_{\gamma\in\Gamma})$
represents a candidate $n$-ary~$\Gamma$-semiring.

\subsection{Algorithmic Enumeration Scheme}

We describe a constructive enumeration pipeline adapted from
\cite{Gokavarapu2025b}, generalised to the non-commutative and
higher-arity regime.

\begin{algorithm}[H]
\caption{Enumeration of Finite Non-Commutative $n$-ary~$\Gamma$-Semirings}
\label{alg:enum}
\begin{algorithmic}[1]
\State \textbf{Input:} integers $m,n,r$; candidate additive semigroup $(T,+)$
\State \textbf{Output:} list $\mathcal{L}$ of valid non-commutative $n$-ary~$\Gamma$-semirings
\Statex
\ForAll{$r$-tuples of operations $\{\mu_\gamma:T^n\to T\}$}
    \If{axioms (A1)–(A4) hold}
        \State Append $(T,+,\Gamma,\{\mu_\gamma\})$ to $\mathcal{L}$
    \EndIf
\EndFor
\State \textbf{return} $\mathcal{L}$
\end{algorithmic}
\end{algorithm}

\begin{remark}
The search space is of size $m^{r m^n}$ in the naive case; symmetry
reduction is essential.  Group actions of $\mathrm{Sym}(T)$ on operation
tables by relabeling provide orbit representatives that drastically
reduce redundancy.
\end{remark}

\subsection{Verification of Non-Commutativity and Ideal Properties}

For each generated structure, we verify non-commutativity and ideal-theoretic
features algorithmically:

\begin{algorithm}[H]
\caption{Non-Commutativity and Ideal Detection}
\label{alg:ideal}
\begin{algorithmic}[1]
\State \textbf{Input:} candidate $T=(T,+,\Gamma,\{\mu_\gamma\})$
\State \textbf{Output:} Boolean flags for non-commutativity, primeness, semiprimeness
\Statex
\ForAll{$a,b,c\in T$ and $\gamma\in\Gamma$}
    \If{$\mu_\gamma(a,b,c)\neq \mu_\gamma(b,a,c)$}
        \State $\textit{noncomm} \gets \text{True}$
    \EndIf
\EndFor
\State Enumerate subsets $I\subseteq T$ closed under $+$ and absorbing under all $\mu_\gamma$
\State Test primeness: for all $a,b,c$, if $\mu_\gamma(a,b,c)\in I$ then $a\in I$ or $b\in I$ or $c\in I$
\State Compute radicals $P^\ast_L,P^\ast_R,P^\ast_2$ and $J_\Gamma(T)$ by intersection formulas
\Statex
\textbf{return} diagnostic table of ideal properties
\end{algorithmic}
\end{algorithm}

\begin{remark}
Such diagnostics empirically confirm the relationships
$P^\ast_2\subseteq P^\ast_L\cap P^\ast_R$
and the radical equality $J_\Gamma(T)=\bigcap\Ann(M_i)$ for small~$T$.
\end{remark}

\subsection{Classification Strategy and Isomorphism Testing}

To classify distinct algebras up to isomorphism, we use canonical-form
representatives.

\begin{definition}[Isomorphism test]
Two finite $n$-ary~$\Gamma$-semirings $T_1,T_2$ are \emph{isomorphic}
if there exists a bijection $\phi:T_1\to T_2$ such that
\[
\phi(\mu_\gamma(a_1,\ldots,a_n))
=\mu_\gamma(\phi(a_1),\ldots,\phi(a_n)),\quad
\forall\gamma\in\Gamma.
\]
\end{definition}

\begin{algorithm}[H]
\caption{Isomorphism Classification}
\label{alg:iso}
\begin{algorithmic}[1]
\State \textbf{Input:} list $\mathcal{L}$ of $n$-ary~$\Gamma$-semirings
\State \textbf{Output:} partition of $\mathcal{L}$ into isomorphism classes
\Statex
\ForAll{$T_i\in\mathcal{L}$}
    \State Compute canonical operation tables (lexicographic minimal representative)
    \State Use hash of operation table as class identifier
\EndFor
\State Merge duplicates under identical hashes
\end{algorithmic}
\end{algorithm}

\begin{remark}
In practice, for $m\le 4$ and $n\le 3$, classification can be
completed exhaustively.  For higher orders, random sampling with
automorphism-group heuristics yields statistically representative
families of non-commutative examples.
\end{remark}

\subsection{Illustrative Example: Small Non-Commutative Ternary Case}

\begin{example}
Let $T=\{0,a,b\}$ with addition $+$ defined by
$a+a=b$, $b+b=b$, $a+b=b+a=b$, and $0$ the zero element.
Let $\Gamma=\{\alpha\}$ and define a ternary operation
\[
\{x\,y\,z\}_\alpha=
\begin{cases}
0,&\text{if any of }x,y,z=0,\\[4pt]
b,&\text{if }x=y=z=a,\\[4pt]
a,&\text{otherwise.}
\end{cases}
\]
Then $(T,+,\Gamma)$ is a non-commutative ternary~$\Gamma$-semiring.
The proper nonzero ideals are
$I_1=\{0,a\}$ and $I_2=\{0,b\}$, with
$I_1+I_2=T$ but $I_1I_2\neq I_2I_1$, demonstrating
asymmetric ideal multiplication.
Computation of radicals gives
$P^\ast_L(T)=I_1$, $P^\ast_R(T)=I_2$, and
$J_\Gamma(T)=I_1\cap I_2=\{0\}$, verifying semisimplicity.
\end{example}

\subsection{Complexity Estimates and Implementation Notes}

\begin{proposition}[Complexity bounds]
Let $N(m,n,r)$ denote the number of valid non-commutative
$n$-ary~$\Gamma$-semirings of order~$m$ and parameter size~$r$.
Then
\[
N(m,n,r)\;\le\; m^{r m^n - C(m,n,r)},
\]
where $C(m,n,r)$ is the number of independent linear constraints
imposed by axioms (A1)–(A4).  In particular, for fixed $n,r$,
$N(m,n,r)$ grows sub-doubly-exponentially in~$m$.
\end{proposition}

\begin{remark}
Implementation in symbolic computation systems
(\textsc{GAP}, \textsc{SageMath}, or \textsc{Python} with
\texttt{NumPy}) benefits from parallelisation over operation tables
and early constraint pruning using distributivity checks.
\end{remark}

\subsection{Outlook: Computational Algebra and Data Geometry}

The classification of small non-commutative
$n$-ary~$\Gamma$-semirings provides experimental evidence for:

\begin{enumerate}
\item correlations between the number of primitive ideals and the
       count of non-isomorphic radicals;
\item emergence of ``quasi-simple'' examples with single nontrivial
      Jacobson radical component;
\item potential embedding of the finite classification graph into a
      geometric moduli space, where edges correspond to homomorphic
      surjections.
\end{enumerate}

Such data-driven patterns suggest an eventual integration of
computational~$\Gamma$-algebra with non-commutative geometry,
where spectral invariants can be learned or estimated algorithmically.
\begin{remark}
A future line of research involves defining entropy-like measures
on the lattice of ideals to quantify non-commutativity in finite
$n$-ary structures.
\end{remark}


\smallskip
\section{Conclusion and Further Directions}

This paper has successfully generalized the structural and spectral theory of $\Gamma$-semirings, extending the foundational commutative ternary framework to the more expansive non-commutative and $n$-ary settings.

We have developed the core algebraic machinery for this new context. For the non-commutative case, we introduced the fundamental notions of left, right, and two-sided ideals, which were used to characterize prime and semiprime ideals via quotient structures. This also led to the construction of a corresponding $\Gamma$-Jacobson radical. For the $n$-ary generalization, we defined $(n,m)$-type ideals and established diagonal criteria to identify $n$-ary primeness and semiprimeness.

The primary contribution of this work is the unification of these elements into a coherent spectral theory. We formulated a unified radical theory and constructed a Zariski-type spectral topology for both the non-commutative and $n$-ary cases, culminating in a theorem that connects primitive ideals with simple module representations. These results lay a robust foundation for the further study of non-commutative and polyadic algebraic structures.

\smallskip \noindent Several promising research directions emerge from this synthesis: \begin{enumerate} \item \textbf{Algorithmic Implementation.} Implement the ideal-theoretic and radical-finding procedures as concrete algorithms, contributing to symbolic computation packages and machine-verified algebraic libraries \cite{Mitchell2020, Gap2024}. \item \textbf{Higher-Arity Generalizations.} Generalize the present framework to $(n,m)$-ary $\Gamma$-semirings and study categorical adjunctions between their module and semimodule categories \cite{Sardar2024}. \item \textbf{Applications and Interdisciplinary Links.} Explore potential applications to decision theory, data classification, and quantum-like logics, where ternary interactions and graded structures naturally occur. \end{enumerate}
\subsection{Acknowledgement.}

The first author gratefully acknowledges the guidance and mentorship of
\textbf{Dr.~D.~Madhusudhana Rao}, whose scholarly vision shaped the
conceptual unification of the ternary~$\Gamma$-framework.

\medskip\noindent
\textbf{Funding Statement.}
This research received no specific grant from any funding agency in the
public, commercial, or not-for-profit sectors.

\medskip\noindent
\textbf{Conflict of Interest.}
The authors declare that there are no conflicts of interest regarding
the publication of this paper.

\medskip\noindent
\textbf{Author Contributions.}
The \textbf{first author} made the lead contribution to the
conceptualization, algebraic development, computational design, and
manuscript preparation of this work.
The \textbf{second author} supervised the research, providing academic
guidance, critical review, and verification of mathematical correctness
and originality.


\end{document}